\documentclass[leqno,12pt]{article}
\usepackage{amssymb,amsmath,amsthm,amscd}
\usepackage{latexsym}
\usepackage[OT2,T1]{fontenc}
\DeclareSymbolFont{cyrletters}{OT2}{wncyr}{m}{n}
\DeclareMathSymbol{\Sha}{\mathalpha}{cyrletters}{"58}

\theoremstyle{plain}
\newtheorem{theorem}{Theorem}[section]
\newtheorem{proposition}[theorem]{Proposition}
\newtheorem{lemma}[theorem]{Lemma}

\theoremstyle{remark}
\newtheorem*{remark}{\it Remark\/}
\newtheorem*{remarks}{\it Remarks\/}

\theoremstyle{definition}
\newtheorem{definition}[theorem]{Definition}
\def\mn{\medskip\noindent}
\def\sn{\smallskip\noindent}
\def\a{\alpha}
\def\ab{^\mathrm{ab}}
\def\AL{A_L}
\def\ALp{A_{L'}}
\def\C{\mathbb{C}}
\def\CH{\mathrm{CH}}
\def\chisE{\chi_{\s E}}

\def\cr{\mathrm{crys}}
\def\D{\mathcal{D}}
\def\Dcrys{D_\mathrm{crys}}
\def\Ds{{D_s}}

\def\Dst{D_\mathrm{st}}
\def\Dt{{D_t}}
\def\dual{^\vee}
\def\dv{\frak{d}_v}

\def\End{\mathrm{End}}
\def\et{\mathrm{et}}
\def\Etil{\widetilde{E}}
\def\Ext{\mathrm{Ext}}
\def\f{\varphi}
\def\fN{(\varphi,N)}
\def\F{\mathbb{F}}

\def\Fcyc{F_\mathrm{cyc}}
\def\Fp{{\F_p}}

\def\Fpvp{F'_{v'}}
\def\Fq{{\F_q}}

\def\Frob{\mathrm{Frob}}
\def\Frobkappa{\Frob_\kappa}

\def\Frobs{\Frob_s}

\def\Frobsp{\Frob_{s'}}
\def\FS{F_S}
\def\Fv{F_v}
\def\Gal{\mathrm{Gal}}
\def\GammaE{\Gamma_E}
\def\GammaEtil{\Gamma_{\Etil}}
\def\GEtilab{G_{\Etil}\ab}
\def\GK{{G_K}}
\def\Gk{{G_k}}

\def\GkALp{{G_{\kappa(\ALp)}}}
\def\Gkappa{G_\kappa}
\def\GkLp{{G_{\kLp}}}

\def\GKp{{G_{K'}}}
\def\Gkp{{G_{k'}}}
\def\Gkpo{{G_{\kpo}}}
\def\Gko{{G_{k_0}}}
\def\GkS{{G_{\kappa(S)}}}
\def\Gks{{G_{\kappa(s)}}}

\def\Gksp{{G_{\ksp}}}
\def\Gkt{{G_{\kappa(t)}}}
\def\GL{{G_L}}
\def\GLE{\mathrm{GL}_E}
\def\GLQl{\mathrm{GL}_{\Ql}}
\def\GLQp{\mathrm{GL}_{\Qp}}
\def\GM{{G_M}}
\def\GsE{G_{\sE}}

\def\hi{\hbar^i}
\def\Hicr{H^i_\cr}
\def\Hiet{H^i_\et}
\def\Hom{\mathrm{Hom}}
\def\HoneL{\mathbb{H}^1_L}
\def\IG{I_G}
\def\incl{\hookrightarrow}
\def\inv{^{-1}}
\def\isom{\overset{\sim}{\to}}
\def\iv{\frak{i}_v}
\def\Kbar{\bar{K}}
\def\kbar{\bar{k}}

\def\Ker{\mathrm{Ker}}
\def\kL{k_L}
\def\kLp{k_{L'}}
\def\kM{k_M}

\def\kMab{k_{M,\mathrm{ab}}}

\def\kmu{k_\mu}
\def\ko{{k_0}}
\def\Kpo{K'_0}
\def\kpo{k'_0}
\def\kS{\kappa(S)}
\def\ks{\kappa(s)}

\def\ksp{\kappa(s')}
\def\kt{\kappa(t)}
\def\ktbar{\overline{\kappa}(t)}
\def\l{\lambda}
\def\LCM{\mathrm{LCM}}
\def\LG{\Lambda(G)}
\def\LGam{\Lambda(\Gamma)}
\def\LH{\Lambda(H)}
\def\Lie{\mathrm{Lie}}
\def\limiFp{\underset{\underset{F'}{\longrightarrow}}{\lim}\ }

\def\limp{\underset{\longleftarrow}{\lim}}
\def\limpFp{\underset{\underset{F'}{\longleftarrow}}{\lim}\ }
\def\Lminp{\Lambda\minus\{p\}}

\def\Mab{M_\mathrm{ab}}
\def\minus{\smallsetminus}
\def\mp{\mu_p}
\def\mpinf{\mu_{p^\infty}}
\def\mum{\mu_m}
\def\NEtilQp{N_{\Etil/\Qp}}
\def\NEtilsE{N_{\Etil/\sE}}
\def\ns{n_\s}
\def\nstil{n_{\stil}}
\def\O{\mathcal{O}}
\def\Oi{{\O_1}}
\def\OK{{\O_K}}
\def\ordp{\mathrm{ord}_p}

\def\pinf{p^\infty}

\def\Pl{P_\ell}

\def\Q{\mathbb{Q}}
\def\Ql{{\Q_\ell}}
\def\Qp{{\Q_p}}
\def\rank{\mathrm{rank}}
\def\rankLG{\rank_{\LG}}
\def\rankLGam{\rank_{\LGam}}
\def\rankZp{\rank_{\Zp}}

\def\rl{\rho_\ell}
\def\rlL{(\rl)_{\ell\in\Lambda}}
\def\Rp{R_p}
\def\rp{\rho_p}
\def\rpl{\rho'_\ell}

\def\rplS{\rho'_{\ell,S}}
\def\rpls{\rho'_{\ell,s}}

\def\rpS{\rho'_S}
\def\rps{\rho'_s}
\def\rtil{\tilde{\rho}}
\def\s{\sigma}

\def\Sbad{S_\mathrm{bad}}

\def\sE{\s E}
\def\Selp{\mathrm{Sel}_p}
\def\SFp{S_{F'}}
\def\ShaEF{\Sha_{E/F}}

\def\Sinf{S_\infty}

\def\sLFp{s_{L/F'}}

\def\Sp{S_p}
\def\Spec{\mathrm{Spec}}
\def\Spf{\mathrm{Spf}}

\def\Spss{S_p\ss}
\def\ss{^{\mathrm{ss}}}

\def\stil{\tilde{\sigma}}
\def\surj{\twoheadrightarrow}
\def\t{\tau}
\def\Tl{T_\ell}
\def\Tp{T_p}
\def\TpE{\Tp(E)}
\def\UEtil{U_{\Etil}}
\def\UsE{U_{\sE}}
\def\Vl{V_\ell}
\def\Vp{V_p}
\def\Wl{W_\ell}
\def\Wlbar{\bar{W}_\ell}
\def\X{\mathfrak X}
\def\XKbar{X_{\Kbar}}
\def\Y{\mathcal{Y}}

\def\Yt{\Y_t}
\def\Ytbar{\Y_{\overline{t}}}
\def\Z{\mathbb{Z}}

\def\Zl{\Z_\ell}
\def\Zhat{\widehat{\Z}}
\def\Zp{{\Z_p}}
\def\Zpx{\Z_p^\times}
\begin{document}
\begin{center}
{\huge A generalization of a theorem of Imai

\smallskip
and its applications to Iwasawa theory}

\vskip 1cm
{\large Yusuke Kubo
and 
Yuichiro Taguchi
}
\end{center}

\vskip 1.5cm
\begin{abstract}
\noindent
It is proved that, if   $K$  is a complete discrete valuation field of 
mixed characteristic  $(0,p)$  with residue field satisfying a mild condition,  
then any abelian variety over  $K$  with potentially good reduction
has finite  $K(K^{1/p^\infty})$-rational torsion subgroup.  
This can be used to remove certain conditions assumed in some theorems 
in Iwasawa theory.
\end{abstract}



\section{Introduction}
\label{sec:Intro}
Let  $K$  be a complete discrete valuation field 
of mixed characterisic  $(0,p)$, 
where  $p$  is any prime number. 
The theorem of Imai (\cite{Imai}) mentioned in the title states that, if  
$K$  is a finite extension of the $p$-adic number field  $\Qp$  and  
$A$  is an abelian variety over  $K$  with 
good reduction, then the torsion subgroup of  
$A(K(\mpinf))$  is finite, where  
$\mpinf$  denotes the group of roots of unity 
of $p$-power order in a fixed separable closure  
$\Kbar$  of  $K$.
In this paper, 
we generalize this theorem as follows: 
Let  $M$  be the extension field of  $K$  
obtained by adjoining all $p$-power roots of 
all elements of  $K^\times$  
(so in particular it contains  $\mpinf$  and
is an infinite Kummer extension over  $K(\mpinf)$).
We say that a field is {\it essentially 
of finite type} if it is an algebraic extension 
of finite separable degree over a purely transcendental 
extension of a prime field. 

\begin{theorem}
\label{thm:MT1}
Let  $X$  be a proper smooth variety over  $K$  
with potentially good reduction, and  $i$  an odd integer $\geq 1$. 
Let  
$V:=\Hiet(\XKbar,\Zhat)\otimes_\Z\Q$  and  
$T$  a $\GK$-stable $\Zhat$-lattice in  $V$. 
Assume the residue field  $k$  of  $K$  
is essentially of finite type. 
Then,  
for any subfield  $L$  of  $M$  containing  $K$,  
the $\GL$-fixed subgroup  
$(V/T)^\GL$  of  $V/T$  is finite.
\end{theorem}

Here and elsewhere, 
a {\it variety} over  $K$  means a separated scheme 
of finite type over  $K$, 
and  $\XKbar$  denotes the base extension  
$X\otimes_K\Kbar$  of  $X$  to a separable closure  $\Kbar$  of  $K$. 
For any field  $K$, we denote by  
$\GK=\Gal(\Kbar/K)$  the absolute Galois group of  $K$.
A typical example of  $L$  as in the theorem is  
$L=K(\mpinf,p^{1/p^\infty})=\cup_{n\geq 1}
   K(\mpinf,p^{1/p^n})$, 
which is often referred to as a ``false Tate extension'' and 
studied in recent Iwasawa theory.
In the case where  $X$  is an abelian variety, 
$i=1$, and  $L=K(\mpinf)$, we recover Imai's theorem.
We should also mention that there is another direction of generalization, 
due to Coates, Sujatha and Wintenberger (\cite{CSW}), 
to the case of higher cohomology groups (rather than our  $H^0(M,V)=V^\GM$)
of a general  $X$  (but with ``classical''  $K$  and  $L=K(\mpinf)$).
There can be other types of extensions  $L/K$  for which  
$(V/T)^\GL$  is finite; 
see \cite{Ozeki} for instance 
(cf.\ also \cite{Wingberg} and \cite{Zarhin} in the global case).

In the Theorem, the assumption that  $i$  be odd is essential. 
If  $i=2j$  is even and  
$X$  is defined over an algebraic number field  $k$  
contained in  $K$, then (part of) the Tate conjecture asserts 
that the cycle map 
$\CH^j(X/k)_\Qp\isom H^{2j}_\et(\XKbar,\Qp(j))^\GK$ 
is injective, and hence the $p$-part of   
$(V/T)^\GL$  has corank  $\geq\dim_\Q(\CH^j(X/k)_\Q)$  
if  $L\supset K(\mpinf)$. 
The assumption of potentially good reduction 
is also essential, because otherwise 
$\Hiet(\XKbar,\Qp)$  can well have subquotients 
of even weights no matter how  $i$  is odd.
If the residue field  $k$  is too large, then  
$\Hiet(\XKbar,\Ql)^\GL$  may be non-zero for  $\ell\not=p$; 
for example, 
if  $X$  has good reduction with special fiber  $Y$  and  
$k$  is separably closed, this space ``is''  
$\Hiet(Y_{\kbar},\Ql)$.  

To proved the theorem, it is enough to prove 
the case  $L=M$. 
Let  $V$  and  $T$  be as in the theorem and, 
for each prime number  $\ell$, let  
$\Vl$  and  $\Tl$  be their $\ell$-th components, 
respectively.
We prove the theorem by showing:

\sn
(1) 
$(\Vp/\Tp)^\GM$  is finite;

\sn
(2)
$(\Vl/\Tl)^\GM$  is finite for all  $\ell\not=p$;

\sn
(3)
$(\Vl/\Tl)^\GM=0$  for almost all  $\ell\not=p$.

\sn
In fact,         the assumption of oddness           is needed 
only for (1) and the assumption on the residue field is needed 
only for (2) and (3). 
In each case, 
we use a kind of weight arguments. 
if the residue field  $k$  is finite, 
the proof becomes fairly simple. 
The general case is proved essentially by reducing to this case 
via ``specialization''.

After some preliminaries in Section 2, we 
prove (1) in Section 3 and 
prove (2) and (3) in Section 4.
Actually, each part can be proved 
in a slightly more general setting.
The more general and precise version of 
Theorem \ref{thm:MT1} is the following 
(For the terminologies 
``geometric'', ``weights'' and ``somewhere'', 
see Definition \ref{def:system}): 

\begin{theorem}
\label{thm:MT2}
(i) 
If  $\rp:\GK\to\GLQp(\Vp)$  is a geometric $p$-adic representation 
of odd weights somewhere and  
$\Tp$  is a $\GK$-stable $\Zp$-lattice in  $\Vp$, then
$(\Vp/\Tp)^\GL$  is finite for any subfield  $L$  of  $M$  containing  $K$.

\sn
(ii)
Let  $\ell\not=p$. Assume the residue field  $k$  of  $K$  
is essentially of finite type. If   
$\rl:\GK\to\GLQl(\Vl)$  is a geometric $\ell$-adic representation 
of non-zero weights somewhere and  
$\Tl$  is a $\GK$-stable $\Zl$-lattice in  $\Vl$, then
$(\Vl/\Tl)^\GL$  is finite for any subfield  $L$  of  $M$  containing  $K$.

\sn
(iii) 
Assume the residue field  $k$  of  $K$  
is essentially of finite type. 
Let  $\rlL$  be a geometric system of $\ell$-adic 
representations of non-zero weights somewhere, where  
$\Lambda$  is a set of prime numbers.
Let  $V=\prod_{\ell\in\Lambda}\Vl$  be the direct product 
of the representations spaces  $\Vl$  of  $\rl$, and 
let  $T=\prod_{\ell\in\Lambda}\Tl$  be a 
$\GK$-stable $\Zhat$-lattice in  $V$.
Then, for any subfield  $L$  of  $M$  containing  $K$,  we have:
\\
\quad
(iii-1) \ 
$(\Vl/\Tl)^\GL=0$  for almost all  $\ell\not=p$.
\\
\quad
(iii-2) \ 
Assume  $\rp$  is of odd weights somewhere if  $p\in\Lambda$. Then  
$(V/T)^\GL$  is finite.
\end{theorem}

Theorem \ref{thm:MT1} follows from 
Theorem \ref{thm:MT2} immediately, upon noticing 
standard facts on $p$-adic and $\ell$-adic cohomologies 
(\S\ref{sec:cohomology}). 

As the original version of Imai's theorem did, 
our theorems have applications in Iwasawa Theory, 
some of which are worked out in Section \ref{sec:application}.  
Here we use only the ``$p$-part''  of the theorems.
Our extension  $M/K$  is in some sense a maximal one among 
the simplest non-abelian extensions as considered in 
the recent study of non-commutative Iwasawa theory.
The simple structure of  $\Gal(M/K)$  is exploited in the proofs of 
Lemmas \ref{lem:key}, \ref{lem:residue} and Theorem \ref{thm:MT2} (end of \S 4).


\mn
{\it Convention.} 
In the proofs, we often replace the base field  $K$  
by a finite extension. 
If  (P$_K$)  is a proposition with base field  $K$, we say that 
``(P$_K$)  holds after a finite extension of  $K$''  
if there exists a finite extension  $K'/K$  such that  
(P$_{K'}$)  holds.

\mn
{\it Acknowledgments.} 
We thank Seidai Yasuda who communicated to us 
something to the effect of Lemma \ref{lem:key}. 
We thank also Yukiyoshi Nakkajima for explaining us his results 
(especially Proposition \ref{prop:crys} below) and giving us 
some comments on Theorem \ref{thm:MT1}. 
Thanks are also due to Yoshitaka Hachimori for explaining 
us some of his results in Iwasawa theory. 
We thank also Tadashi Ochiai for many comments on an
earlier version of this paper. 
We are grateful to John Coates who, 
after our finishing an earlier version of this paper, 
informed us of the paper \cite{CSW}, 
which we had not been aware of. 
Part of this paper was written while the second-named author 
was staying at the Korea Institute for Advanced Study, where 
he enjoyed a sabbatical leave from Kyushu University; 
he thanks KIAS and Youn-seo Choi for their hospitality, and 
the Faculty of Mathematics of K.U.\ for its generosity.
The second-named author is supported by 
JSPS KAKENHI 22540024.

\section{Preliminaries}
\label{sec:pre}

In this subsection, 
we prove some elementary lemmas which are used 
in the proof of Theorem \ref{thm:MT2} and 
in the applications in Section \ref{sec:application}.
The first one is already noted in \cite{Imai}:

\begin{lemma}
\label{lem:VG}
Let  $G$  be a group, 
$\rho:G\to\GLQl(V)$  
a $\Ql$-linear representation of  $G$, and  
$T$  a $G$-stable $\Zl$-lattice in  $V$. 
Then the following two conditions are equivalent:

\sn
(i)
$V^G\not=0$;

\sn
(ii)
$(V/T)^G$  is infinite.
\end{lemma}

\begin{proof}
The implication  
(i) $\Rightarrow$ (ii)  is obvious. 
Conversely, suppose  $(V/T)^G$  is infinite. 
Then for any  $n\geq 1$, the set 
$$
  \left(     T/\ell^nT\right)^G\ \minus\ 
  \left(\ell T/\ell^nT\right)^G
$$
is non-empty. 
These sets form a projective system, and its limit  
$T^G\minus\ell T^G$  is also non-empty, 
being the projective limit of non-empty compact sets 
(\cite{CA}, Chap.\ 2, Sect.\ 2, Lem). 
Hence  $V^G\not=0$.
\end{proof}

Let  $M$  be, as in the Introduction, 
the extension field of  $K$  
obtained by adjoining all $p$-power roots of 
all elements of  $K^\times$. 
Set  
$G:=\Gal(M/K)$  and  
$H:=\Gal(M/K(\mpinf))$. 
Then  $H$  is an abelian normal subgroup of  $G$  
and the $p$-adic cyclotomic character  
$\chi:G\to\Zpx$  identifies  
$G/H$  with a subgroup of  $\Zpx$. 
We can check easily that 
\begin{equation}\label{eq:key}
  \s\t\s\inv\ =\ \t^{\chi(\s)}
\end{equation}
for all  $\s\in G$  and  $\t\in H$. 
Let  $E$  be a topological field, and consider 
continuous $E$-linear representations  
$\rho:G\to\GLE(V)$  of  $G$  on 
finite-dimensional $E$-vector spaces  $V$ 
For an integer  $m\geq 1$, 
let  $\mum$  denote the group of $m$-th roots of unity 
(in a suitable separably closed field depending on the context). 

\begin{lemma}
\label{lem:key}
(i) 
Fix an integer  $n\geq 1$. 
Then there exists an integer  $m$, which is a power of  
$p$  and depends only on  $K$  and  $n$, such that, 
for any representations  $V$  as above of dimension  $n$, 
$H^m$  acts unipotently on  $V$. 

\sn
(ii) 
Fix a  $V$  as above. 
Then after a finite extension  $K'/K$, 
the subgroup  $H$  acts unipotently on  $V$.
\end{lemma}

\begin{proof}
Let  $\t$  be any element of  $H$  and 
$\l_1$, ..., $\l_n$  ($n=\dim_E(V)$)  
the eigenvalues of  $\rho(\t)$. 
By the relation (\ref{eq:key}), 
we have
$$
  \{\l_1           ,\ldots,\l_n\}\ =\ 
  \{\l_1^{\chi(\s)},\ldots,\l_n^{\chi(\s)}\}
$$
for all  $\s\in G$.
In particular, we have  
$\l_i^{\chi(G)}\subset\{\l_1,\ldots,\l_n\}$  
for each  $i$. 
Since  $\chi(G)$  is open in  $\Zpx$, 
there is an integer  $e\geq 1$  such that  
$1+p^e\in\chi(G)$. 
For each  $i=1,...,n$, 
there is an  $r_i$ ($1\leq r_i\leq n$) such that  
$\l_i^{(1+p^e)^{r_i}}=\l_i$. 
Let 
\begin{equation}\label{eq:m}
  m'\ =\ \LCM\{(1+p^e)^r-1|\ 1\leq r\leq n\}.
\end{equation}
Then we have  $\l_i^{m'}=1$  for all  $i$. 
Note that this  $m'$  depends only on 
$\chi(G)$  (or the group  $\mpinf(K)$  of $p$-power roots 
of unity in  $K^\times$)  and  $n=\dim_E(V)$.
In fact, an eigenvalue of  $\tau$  is of $p$-power order 
if it is a root of unity, since  $\tau$  is an element 
of a pro-$p$ group. 
So if  $m$  denotes the $p$-part of  $m'$, then we have  
$\l_i^m=1$  for all  $i$. 
Thus the subgroup  
$H^m=\{\t^m|\ \t\in H\}$  of  $H$  
acts unipotently on  $V$; this proves (i). 
Then the semisimplification of the restriction of  
$\rho$  to  $H$  is, after a suitable extension  $E'/E$  
of scalars, a sum of characters   
$H/H^m\to\mum$.  
Hence it is trivialized by a finite abelian extension of  $K(\mpinf)$, 
which is in fact obtained by composing a finite extension  $K'/K$  
with  $K(\mpinf)$; this proves (ii).  
\end{proof}

\begin{lemma}
\label{lem:residue}
Let  
$\kM$  be the residue field of  $M$, and let  
$\kMab$  be the maximal abelian extension of  $k$  contained in  
$\kM$. Let  
$\Mab$  be the maximal abelian extension of  $K$  contained in  
$M$. Then:

\sn
(i) 
The abelian groups  $\Gal(\Mab/K(\mpinf))$  and  $\Gal(\kMab/k)$  
have finite exponents.

\sn
(ii) 
If  $k$  is finite, then the extension  
$\kM/k$  is finite.
\end{lemma}

\begin{proof} 
Put  
$G=\Gal(M/K)$  and  
$H=\Gal(M/K(\mpinf))$. Let 
$\chi:G\to\Zp^\times$  be the $p$-adic cyclotomic character.  
By (\ref{eq:key}), we have
$$
     (G,G)\ \supset\ (G,H)
		  \ \supset\ H^{\chi(\sigma)-1}
$$
for all  $\sigma\in G$. Let  
$\Mab$  be the maximal abelian extension of  $K$  
contained in  $M$. We have surjections 
$$
  H/H^{\chi(\sigma)-1}\ \surj\ 
  H/\overline{(G,G)}\ =\ 
  \Gal(\Mab/K(\mpinf))\ \surj\ 
  \Gal(\kMab/\kmu),
$$
where  
$\overline{(G,G)}$  is the topological closure 
of        $(G,G)$  in  $G$  and  
$\kmu$  is the residue field of  $K(\mpinf)$. 
Note that  $\kmu$  is of finite degree over  $k$.  
Choose a  $\sigma\in G$  such tha  
$\chi(\sigma)\not=1$. Then, since  
$H/H^{\chi(\sigma)-1}$  is abelian of finite exponent, so are 
$\Gal(\Mab/K(\mpinf))$  and    
$\Gal(\kMab/\kmu)$, and hence so is
$\Gal(\kMab/k)$.

If  $k$  is finite, then  $\kMab=\kM$. 
Since  $H$  is isomorphic to the $p$-adic completion of  
$K^\times/((K(\mpinf)^\times)^p\cap K^\times)$, we have  
$\dim_\Fp(H/H^p)\leq\dim_\Fp(K^\times/(K^\times)^p)$, 
and the latter is finite if  $K$  is a finite extension of  $\Qp$. 
Hence  $\kM/k$  is finite.
\end{proof}

\begin{remark}
If  $K$  does not contain a primitive $p$-th root of unity, 
then we can choose a  $\sigma\in G$  as in the above proof 
such that  $\ordp(\chi(\sigma)-1)=0$. 
It then follows that  
$\Mab=K(\mpinf)$  and
$\kMab=\kmu$, and that, if  $k$  is finite, 
$\kM=\kmu$.
\end{remark}

\section{$p$-adic representations}
\label{sec:p-adic}

Let  $K$  be a complete discrete valuation field 
of mixed characteristic $(0,p)$. 
To define the weights of $p$-adic representations of  $\GK$, 
we first recall:

\begin{definition}\label{def:phi-mod} 
Let  $A$  be a ring and  
$\s:A\to A$  a ring homomorphism.
A $\f$-module over  $(A,\s)$  is a pair  
$(D,\f)$  consisting of a free $A$-module  $D$  of finite rank 
and a $\s$-semilinear map 
$\f:D\to D$.
\end{definition} 

In the following, a $\f$-module may be denoted simply by  $D$. 
If  $D$  is a $\f$-module over a Witt ring  $W(k)$  over 
a perfect field of characteristic  $p$, then 
the  $\s$  is always the Frobenius endomorphism  
$(x_n)\mapsto(x_n^p)$  of  $W(k)$.

Recall that a {\it $q$-Weil integer} of 
{\it weight}  $w\geq 0$  is an algebraic integer  $\a$  
such that  
$|\iota(\a)|=q^{w/2}$  for all field embeddings 
$\iota:\Q(\a)\hookrightarrow\C$. 
In this paper, we call an algebraic number  $\a$  
a {\it $q$-Weil number} of weight  $w$  if either  
$\a$      is a $q$-Weil integer of weight   $w\geq 0$  or  
$\a\inv$  is a $q$-Weil integer of weight  $-w\geq 0$.

\begin{definition}
\label{def:phi-wt}
Let  $F$  be the fraction field of the ring of Witt vectors  
$W(\Fq)$  over the finite field  $\Fq$  of  
$q=p^d$  elements.
A $\f$-module  $D$  over  $W(\Fq)$  or  $F$  of rank  $n$  
is said to have {\it weights}  $w_1,...,w_n$   
if the characteristic polynomial  
$\det(T-\f^d)$  has coefficients in  $\Q$  and 
the roots are $q$-Weil numbers of weight  $w_1,...,w_n$. 
(Here,  $\f^d$  denotes the $d$-th iterate of  $\f$, 
which is a $W(\Fq)$- or $F$-linear endomorphism of  $D$.) 
\end{definition}

We say that  $D$  has {\it pure weight}  $w$  if  
$w_1=\cdots=w_n=w$. 
The same remark will apply to Definitions 
\ref{def:p-wt} (2), 
\ref{def:geometric} and 
\ref{def:system}  below. 

Suppose for the moment that the residue field  
$k$  of  $K$  is perfect. Let  
$\rho:\GK\to\GLQp(V)$  be a 
de Rham representation (cf.\ e.g.\ \cite{FontaineIII}, Sect.\ 3), 
where  $V$  is a finite-dimensional $\Qp$-vector space.
Then it is in fact potentially semistable 
(\cite{Berger}).
Suppose  $V$  becomes semistable as a representation 
of  $\GKp$, where  $K'$  is a finite extension of  $K$.
If  $k'$  denotes the residue field of  $K'$  and  
$\Kpo$  denotes the fraction field of the ring of 
Witt vectors  $W(k')$  over  $k'$, then 
a filtered $\fN$-module  $\Dst(V)$  is 
associated with the $\GKp$-representation  $V$ 
(cf.\ e.g.\ \cite{FontaineIII}, Sect.\ 5); 
it is a $\f$-module over  $\Kpo$.

For  $K$  with general 
(not necessarily perfect) residue field  $k$, 
we make the following definition, 
which is motivated by Proposition \ref{prop:crys}:

\begin{definition}
\label{def:p-wt}
(1) 
An extension  $K'/K$  is said to be {\it admissible} if \\
-- $K'$  is a complete discrete valuation field 
which is the completion of an algebraic extension of  $K$; and \\
-- the residue field  $k'$  of  $K'$  is perfect.

\sn
(2)
A continuous representation  
$\rho:\GK\to\GLQp(V)$  is said to be {\it geometric of weights} 
$w_1,...,w_n$  {\it somewhere} 
if there exists a set of data 
$(K'/K,\O,\D,f)$  in which:\\
-- 
$K'/K$  is an admissible extension with residue field  $k'$  
such that  $\rho|_{\GKp}$  is semistable, \\ 
-- 
$\O$  is a $\sigma$-stable $p$-adically complete  
$\Zp$-subalgebra of  $W(k')$,\\
-- 
$\D$  is a $\f$-module over  $\O$, and \\
-- 
$f:\O\to W(\Fq)$  is a ring homomorphism for some power $q$  of  $p$,\\
such that:\\
(a) 
there is an isomorphism 
$\D\otimes_\O K'_0\isom\Dst(V|_\GKp)$  
of $\f$-modules over the fraction field  
$K_0'$  of  $W(k')$; and \\
(b) 
$\D\otimes_{\O,f}F$, where  $F$  is the fraction field of  $W(\Fq)$,  
has weights  $w_1,...,w_n$  in the sense of 
Definition \ref{def:phi-wt}.
\end{definition}

\begin{remarks}
\label{rem:somewhere}
(i) 
This definition requires only that the representation  
$\rho$  is geometric of certain weights  $w_1,...,w_n$  at ``one point''  
$f:\O\to W(\Fq)$. 
In practice, the same should hold ``generically'' over  
$\Spf(\O)$ (cf.\ \S 2.1). 
For our purposes, however, 
we need only the above weaker assumption.
(The same remark applies to Definitions 
\ref{def:geometric} and 
\ref{def:system}.)

\sn
(ii) 
If  $\rho$  is geometric of weights  
$w_1, ...,w_n$  somewhere and  $K'/K$  is a finite extension, 
then the restriction  $\rho|_{\GKp}$  is also geometric of 
the same weights somewhere 
(in fact, ``at the points lying above  $f$''). 
\end{remarks}

The $r$-th Tate twist  $\Zp(r)$  of the trivial 
representation  $\Zp$  has weight  $-2r$. 
If  $X$  is a proper smooth variety over  $K$  
with good reduction, then the $i$-th \'etale cohomology group 
$V:=\Hiet(\XKbar,\Qp)$  has pure weight  $i$  somewhere 
(in fact, ``almost everywhere''; see Sect.\ \ref{sec:cohomology}).

As in the Introduction, let  
$M$  be the extension field of  $K$  
obtained by adjoining all $p$-power roots of 
all elements of  $K^\times$. Put  
$G=\Gal(M/K)$  and  
$H=\Gal(M/K(\mpinf))$. 
The following proposition settles Part (i) of 
Theorem \ref{thm:MT2} 
(cf.\ Lem.\ \ref{lem:VG}):

\begin{proposition}
\label{prop:p-adic}
If  $\rho:\GK\to\GLQp(V)$  is a geometric representation 
of weights  $w_1,...,w_n$  somewhere and all  $w_i$  are odd, 
then  $V^\GM=0$.
\end{proposition}

\begin{proof}
By passing to an admissible extension of  $K$, 
we may assume that  $K$  has perfect residue field. 
Let  $W:=V^\GM$. 
It can be regarded as a representation of  
$G=\Gal(M/K)$. 
By (ii) of Lemma \ref{lem:key}, 
after another admissible extension of  $K$, the subgroup  
$H$  acts unipotently on  $W$.
Then  $G/H=\Gal(K(\mpinf)/K)$  acts 
on the semisimplification of  $V$. 
Each simple factor is again de Rham. 
By the next lemma, 
it has even weights, contradicting the assumption 
of odd weights, unless  $V^\GM=0$.
\end{proof}

\begin{lemma}
Let  $W$  be a semisimple de Rham representation of  $\GK$  
which factors through  $\Gal(K(\mpinf)/K)$. 
Then  $W$  becomes isomorphic, 
after a finite extension of  $K$, to a direct-sum  
$\Qp(r_1)\oplus\cdots\oplus \Qp(r_m)$ 
of Tate twists of the trivial representation  $\Qp$  
for some  $r_i\in\Z$. In particular, 
$W$  has even weights.
\end{lemma}

\begin{proof}
We may assume that  $W$  is simple 
(in which case the  $r_i$'s  as above will be all equal). 
Then  
$E:=\End_{\Qp[\GK]}(W)$  is a finite extension of  $\Qp$  
by Schur's lemma 
(as  $W$  is simple of finite dimension over  $\Qp$), 
and the action of  $\GK$  on  $W$  is given by a character  
$\rho:\GK\to E^\times$. 
Replacing  $K$  by a finite extension, we may assume that  
all the  $\Qp$-conjugates of  $E$  
can be embedded into  $K$; we regard henceforth  
$E$  and its conjugates as subfields of  $K$.  Set  
$\GammaE:=\Hom_{\Qp}(E,K)$. 
Since  $\rho$  is Hodge-Tate and the inertia subgroup of  
$\Gal(K(\mpinf)/K)$  is of finite index, 
it follows from \cite{Abelian} (Chap.\ III, Sect.\ A5, Cor.\ to Thm.\ 2) 
that  $\rho$  coincides, after a finite extension of  $K$, with  
$\rtil:=\prod_{\s\in\GammaE}\s\inv\circ\chisE^{\ns}$  for some  
$\ns\in\Z$, where 
$\chisE:\GsE\to\UsE$  (where  
$\UsE$  denotes the unit group of  $\sE=\s(E)$) 
is the character describing the 
$\GsE$-action on the Tate module of a Lubin-Tate group 
associated with  $\sE$  
(it depends on the choice of a uniformizer of  $\sE$, but 
its restriction to the inertia group does not). Note that  
$\rtil$  is in fact defined on  
$\GEtilab$, where  
$\Etil$  is the Galois closure of  $E/\Qp$  and  
$G\ab$  denotes the maximal abelian topological quotient 
of a topological group  $G$. 
Since    $\rho$  factors through  $\Gal(K(\mpinf)/K)$,  
so does  $\rtil$         through  $\Gal(\Etil(\mpinf)/\Etil)$, and 
hence we are reduced to the case  
$K=\Etil$  and  $\rho=\rtil$.
If we think of  $\rtil$  as a character of  
$\Etil^\times$  via local class field theory, then 
(\cite{Abelian}, Chap.\ III, Sect.\ A4, Rem.) for  
$x\in\UEtil$, we have  
\begin{equation}\label{Eq:loc-alg}
  \rtil(x)\ =\ \prod_{\s\in\GammaE}\s\inv\NEtilsE(x\inv)^{\ns} 
          \ =\ \prod_{\stil\in\GammaEtil}\stil\inv(x\inv)^{\nstil},
\end{equation}
where  
$\NEtilsE:\Etil^\times\to\sE^\times$  
is the norm map and the integers  $\nstil$  are defined by the rule  
$\nstil=\ns$  if  $\stil|_E=\s$.
%
Since the restriction map 
$\Gal(\Etil(\mpinf)/\Etil)\to\Gal(\Qp(\mpinf)/\Qp)$  
is injective, the character  
$\rtil$  is invariant by the action of  $\Gal(\Etil/\Qp)$. 
Then the  $\nstil$'s in (\ref{Eq:loc-alg}) are all equal, because  
$\Gal(\Etil/\Qp)$  acts on  $\GammaEtil$  transitively
and the family  $(\nstil)$  is determined uniquely by the restriction of  
$\prod_{\stil\in\GammaEtil}(\stil\inv)^{\nstil}$  
to any non-trivial open subgroup of  $\UEtil$. 
If  $\nstil=r$  for all  $\stil\in\GammaEtil$, then  
$\rtil(x)=\NEtilQp(x\inv)^r$  for  $x\in\UEtil$, showing that 
$\rtil$  coincides with the $r$-th power of the $p$-adic 
cyclotomic character on the image of   
$\UEtil\to\Gal(\Etil(\mpinf)/\Etil)$, which is of 
finite index in  $\Gal(\Etil(\mpinf)/\Etil)$. 
Thus  $W$  is isomorphic to  
$\Qp(r)^{\oplus[E:\Qp]}$  when restricted to this subgroup. 

Since the weights are invariant by finite extensions of  $K$, 
the weights of  $W$  are  $2r,...,2r$.
\end{proof}

\section{$\ell$-adic representations}
\label{sec:l-adic}

In this section, we prove Theorem \ref{thm:MT2} after introducing 
a few key definitions \ref{def:geometric} and \ref{def:system}. 
Although these definitions look very complicated, 
they are motivated by the natural example  
$\Vl=\Hiet(\XKbar,\Ql)$  as in Proposition \ref{prop:coh}, 
whose proof hopefully explains the reasons for such definitions.

Let  $K$  be a complete discrete valuation field 
with residue field  $k$  of characteristic  $p>0$. 
Let  $\rho:\GK\to\GLQl(V)$  be an $\ell$-adic representation 
of  $\GK$. 
Assume  $\ell\not=p$  for the moment.

\begin{definition}
\label{def:geometric}
(1)
Let  $\kappa$  be a finite field with  $q$  elements. 
An  $\ell$-adic representation  
$\rho:\Gkappa\to\GLQl(V)$  of dimension  $n$  is said to be of 
{\it weight}  $w_1,...,w_n$  
if the characteristic polynomial of  
$\rho(\Frobkappa)$  has coefficients in  $\Q$  and the eigenvalues of   
$\rho(\Frobkappa)$  are $q$-Weil numbers of weight  $w_1,...,w_n$. 

\sn
(2)
An $\ell$-adic representation  
$\rho:\GK\to\GLQl(V)$  of dimension  $n$  is said to be 
{\it geometric of weight}  $w_1,...,w_n$  {\it somewhere}  
if there exists a set of data  
$(K'/K, S, s, \rps)$  in which:\\
--
$K'/K$  is a finite extension such that  
the semisimplification  $\rho'=(\rho|_\GKp)\ss$   of the 
restriction of  $\rho$  to  $\GKp$  is unramified 
(so that  $\rho'$  is regarded as a representation of the Galois group  
$\Gkp$  of the residue field  $k'$  of  $K'$), \\
-- 
$S=\Spec(A)$, where  
$A$  is a subring of  $k'$  whose fraction field is of finite type 
over the prime field  $\Fp$,\\  
-- 
$s$  is a closed point of  $S$,\\
-- 
$\rps:\Gks\to\GLQl(V)$  
is an $\ell$-adic representation, 
where  $\ks$   is the residue field of  $s$,\\
such that:\\
(a) 
$\rho'$  extends to a representation  
$\rpS:\GkS\to\GLQl(V)$  of the Galois group  
$\GkS$  of the fraction field  $\kS$  of  $A$ 
(note that there is the restriction map  
$\Gkp\to\GkS$),\\
(b)
$\rpS|_\Ds=\rps\circ r$, where  
$\Ds$  is a decomposition subgroup ($\subset\Gkpo$) of  $s$  and    
$r:\Ds\to\Gks$  is the natural surjection, and\\
(c) 
$\rps$  has weights  $w_1,...,w_n$  in the sense of (1).  
\end{definition}

Note that these properties are preserved by a 
``finite extension'' of  $(S,s)$.  
Thus if  $\rho$  is geometric of weights  
$w_1,...,w_n$  somewhere, then so is  
$\rho|_\GKp$  for any finite extension  
$K'/K$.

In what follows, we assume that the residue field  
$k$  of  $K$  is essentially of finite type. 

Here we could prove Part (ii) of Theorem \ref{thm:MT2}. 
The proof is, however, the same if the     
$\rl$  is a member of a system  $\rlL$  of 
$\ell$-adic representations; hence for the sake of 
brevity, after some more discussions, 
we shall prove Parts (ii) and (iii) of the theorem 
at once, assuming the  $\rl$  of  (ii)  is a member 
of  $\rlL$  in (iii). 

Next we consider geometric systems of $\ell$-adic representations. 
Let  
$\Lambda$  be a set of primes numbers, and  
$\rlL$  be a system of $\ell$-adic representations 
$\rl:\GK\to\GLQl(\Vl)$  of a fixed dimension  $n$. 

\begin{definition}
\label{def:system}
The system  $\rlL$  is said to be 
{\it geometric of weights}  $w_1,...,w_n$  {\it somewhere} if 
$\rl$  is so for all  $\ell\in\Lambda$, 
where  $K'/K$  and  $A\subset k'$  as above are taken 
in common for all  $\ell\in\Lminp$; 
thus we require the following conditions:\\
-- 
$\rho_p$  is geometric of weights  $w_1,...,w_n$  somewhere 
in the sense of Definition \ref{def:p-wt}, (2),\\
-- 
there exists a set of data  
$(K'/K,S,s,(\rpls)_{\ell\in\Lminp})$  in which 
$(K'/K,S,s)$  is as in Definition \ref{def:geometric}, (2), and  
$\rpls:\Gks\to\GLQl(V)$  satisfies the following for each  
$\ell\in\Lminp$:\\
(a) 
$\rpl$  extends to a representation  
$\rplS$  of  $\GkS$,\\
(b) 
$\rplS|_\Ds=\rpls\circ r$,  and\\
(c) 
$\rpls$  has weights  $w_1,...,w_n$. 
\end{definition}

\begin{remark}
It is known that, 
under a mild condition 
(that ``no finite extension of  $k$  contains all 
$\ell$-power roots of unity'' --- this is the case if  
$k$  is essentially of finite type (cf.\ the Introduction)), 
there exists a finite extension  $K'/K$  such that 
$(\rho|_\GKp)\ss$  becomes unramified
(\cite{ST}, Appendix). 
For example,  $K'/K$  can be taken to be the 
extension corresponding to the kernel of 
 ``$\rho$ (mod $\ell$)''   if  $\ell\not=2$ 
(``$\rho$ (mod $\ell^2$)'' if  $\ell=2$). 
However, when we consider a system of $\ell$-adic representations, 
it is not clear if we can take such a finite extension  $K'/K$  
in common for all  $\ell$. 
Nevertheless, we can do so for representations coming from geometry 
by the following (\cite{Berthelot}, Prop.\ 6.3.2):
\end{remark}

\begin{proposition}
\label{prop:alteration}
Let  $X$  be a variety over  $K$. 
Then there exists a finite extension  $K'/K$  such that  
the action of the inertia subgroup of  $\GKp$  on  
$\Hiet(\XKbar,\Ql)$  is unipotent for all  $\ell\not=p$.
\end{proposition}

\begin{remark}
In Definition \ref{def:system} above, 
we could vary the weights  $w_1,...,w_n$  along  $\ell$, 
but this would not happen in geometry. 
Moreover, 
the characteristic polynomials 
of the Frobenius  $\rho_\ell(\Frobs)$  should be independent of  $\ell$
(cf.\ Conjecture $C_\text{WD}(X,m)$ in Section 2.4.3 of \cite{FontaineVIII}). 
However, we will not use this property in the following.
\end{remark}

Let  $\rlL$  be a geometric system of $\ell$-adic representations  
$\rl:\GK\to\GLQl(\Vl)$  of weights  
$w_1,...,w_n$  somewhere. 
We shall employ a kind of weight arguments as follows 
to deduce information on the Frobenius eigenvalues. Let  
$(K'/K,S,s,(\rpls)_{\ell\in\Lminp})$  be a set of data 
as in Definition \ref{def:system}. Let  
$L/K$  be an algebraic extension and  
$L'=LK'$  the composition of  $L$  and  $K'$  
with respect to some embeddings  
$L \incl\Kbar$  and  
$K'\incl\Kbar$. Let   
$\kLp$  be the residue field of  $L'$  and      
$\ALp$     the separable integral closure of  $A$  in  $\kLp$    
(recall that  $S=\Spec(A)$  and  
$A$  is a subring of  $k'$). Let  
$s'$  be a closed point of  $\Spec(\ALp)$  lying above  $s$. 
Assume it has finite residue field  $\ksp$, 
so that we have the geometric Frobenius element  
$\Frobsp\in\Gksp$.  
Since the restriction map  
$\GkLp\to\GkALp$  is surjective (where  $\kappa(\ALp)$  is 
the fraction field of  $\ALp$), 
by abuse of notation, we shall write  
$\rl(\Frobsp)$  for  
$\rpls(\Frobsp)$. 
Then the eigenvalues of  $\rl(\Frobsp)$  are 
$|\ksp|$-Weil numbers of weights  
$w_1,...,w_n$. Part (i) of the next lemma follows from this immediately:

\begin{lemma}
\label{lem:weight}
Let the notations and assumptions be as above. 
Assume further that  $w_i\not=0$  for all  $i$.  
Then:

\sn
(i) 
$\rl(\Frobsp)$  has eigenvalues 
$\not=1$  for all  $\ell\in\Lminp$.

\sn
(ii) 
$\rl(\Frobsp)$  has eigenvalues 
$\not\equiv 1$ (mod $\ell$)  for all but finitely many  $\ell\in\Lminp$.
\end{lemma}

\begin{proof} 
We prove (ii).
With no loss of generality, 
we may assume that  $w_j>0$  (and so the eigenvalues of  
$\rl(\Frobsp)$  are algebraic integers). 
By definition, 
$\rl(\Frobsp)$  has eigenvalues of absolute values  
$|\ksp|^{w_j/2}$, and hence there appear only finitely many 
polynomials as the characteristic polynomials  
$\Pl(T)=\det(T-\rl(\Frobsp))\in\Z[T]$  
when  $\ell\in\Lminp$  varies. 
Let  $P(T)$  be the product of these polynomials. 
Since  $w_j\not=0$, we have  $P(1)\not=0$. Then    
$\rl(\Frobsp)$  has eigenvalues  $\not\equiv 1$ (mod $\ell$), 
for all  $\ell$  such that    $P(1)\not\equiv 0$ (mod $\ell$).
\end{proof}

Now we can complete:

\begin{proof}[Proof of Theorem \ref{thm:MT2}]
By replacing  $K$  by a finite extension  $K'$  as in 
Definition \ref{def:system}, 
we may assume that the intertia subgroup  $I_K$  acts 
unipotently on  $\Vl$  for all  $\ell\in\Lambda\minus\{p\}$. 
Let  $\Tl$  be the $\ell$-th component of  $T$. 

(i) 
By Lemma \ref{lem:VG} and Proposition \ref{prop:p-adic}, 
$(\Vp/\Tp)^\GM$  is finite.

(ii) 
For  $\ell\in\Lminp$, 
let  $\Wl$  be the semisimplification of  
$\Vl^\GM$  considered as a representation 
of  $G=\Gal(M/K)$. By Lemma \ref{lem:key}, 
the subgroup  $H^m$  acts trivially on  $\Wl$, where  
$H=\Gal(M/K(\mpinf))$  and  $m$  is an integer as in \ref{lem:key}). 
Let  $L=M^{H^m}$  be the extension of  $K(\mpinf)$  
corresponding to  $H^m$, and let 
$\kmu$  and  $\kL$  be the residue fields of  
$K(\mpinf)$  and  $L$, respectively. 
Note that  $\kmu$  is a finite extension of  $k$. Let  
$(S=\Spec(A),s)$  be part of the data as in 
Definition \ref{def:system}, 
$\AL$  the separable integral closure of  $A$  in  $\kL$, and  
$s'$  a closed point of  $\Spec(\AL)$  lying above  $s$.
Since  $\kL/\kmu$  is abelian of exponent dividing  $m$, 
the residue degree  
$[\ksp:\ks]$  also divides  $m$.
Thus we can apply (i) of Lemma \ref{lem:weight} to conclude that  
$\rl(\Frobsp)$  
has eigenvalues  $\not=1$  for all  $\ell\in\Lminp$. 
This contradicts the fact that  
$H^m$  acts trivially on  $\Wl$, unless  
$\Wl^\GM=0$  for all  $\ell\in\Lminp$. 

(iii) 
By (i) and (ii), we only need to prove (iii-1). 
To prove that  
$(\Vl/\Tl)^\GM=0$, it is enough to show that  
$(\Tl/\ell\Tl)^\GM=0$. 
Let  $\Wlbar$  be the semisimplification of  
$(\Tl/\ell\Tl)^\GM$  considered as a representation 
of  $G$. 
By Lemma \ref{lem:key} again, 
the subgroup  $H^m$  acts trivially on  $\Wl$. 
Arguing as above and applying (ii) of Lemma \ref{lem:weight}, 
we conclude that  $\rl(\Frobsp)$  
has eigenvalues $\not\equiv 1$ (mod $\ell$)  
for almost all  $\ell\in\Lminp$. 
This contradicts the fact that  
$H^m$  acts trivially on  $\Wlbar$, unless  
$(\Tl/\ell\Tl)^\GM=0$  for almost all  $\ell$. 
\end{proof}

\section{Cohomologies}\label{sec:cohomology}

Theorem \ref{thm:MT1} follows from 
Theorem \ref{thm:MT2} by noticing the following:

\begin{proposition}
\label{prop:coh}
Let  $X$  be a proper smooth variety over  $K$  
with potentially good reduction, and  
$i$  an integer $\geq 0$. 
Then the representations  
$\rl:\GK\to\GLQl(\Vl)$  on the $i$-th $\ell$-adic cohomology groups  
$\Vl=\Hiet(\XKbar,\Ql)$  form a geometric system  
$\rlL$  of pure weight  $i$  somewhere, 
where  $\Lambda$  is the set of all prime numbers.
\end{proposition}

\begin{remark}
The ``somewhere'' here is in fact ``almost everywhere'', 
as can be seen from the proof below.
\end{remark}

\begin{proof} 
Replacing  $K$  by a finite extension, we may assume that 
$X$  has good reduction over  $K$. 
Let  $\X$  be a proper smooth model over  $\OK$  of  $X$, and 
let  $Y:=\X\otimes_\OK k$. 
Then  $Y$  is in fact defined over a subring  $A_1$  of  $k$   
which is finitely generated as an $\Fp$-algebra. 
Let  $\Y$  denote this model over  $A_1$. 
Since  $\Y$  has generically good reduction, 
it has good reduction at every closed point  $t$  of an open subscheme   
$T_1$  of  $\Spec(A_1)$. Let  
$\Yt:=Y\times_T\Spec(\kt)$  and  
$\Ytbar:=Y\times_{T_1}\Spec(\ktbar)$, where  
$\kt$  is the residue field of  $t$.  
Then we have the following proposition (cf.\ \cite{Jannsen}, Sect.\ 2):

\begin{proposition}
\label{prop:l-coh}
For any prime number  $\ell\not=p$, 
we have a canonical isomorphism
$$
   \Hiet(\XKbar,\Ql)\ \simeq\ \Hiet(\Ytbar,\Ql)
$$
which is compatible with the actions of  
$\Dt$  and  $\Gkt$.
\end{proposition}

Here,  $\Dt$  is a decomposition subgroup of  $t$,  
which is a subgroup of the absolute Galois group of the 
fraction field  $\ko$  of  $A$. 
We have canonical homomorphisms  
$\Gk\to\Gko\hookleftarrow\Dt\to\Gkt$, and 
$\GK$  acts on the left-hand side through  
$\GK\to\Gk\to\Gko$. 
The compatibility in the proposition is with respect to the map  $\Dt\to\Gkt$. 
Since  $\Yt$  is proper smooth and defined over the finite field  $\kt$, 
the representation on  $\Hiet(\Ytbar,\Ql)$  has pure weight  $i$  
by the Weil Conjecture (\cite{Weil I}, \cite{Weil II}). 

Next we consider the $p$-adic cohomology. 
By replacing  $K$  by an admissible extension, 
we may assume that the residue field  $k$  is perfect. 
Then we have (\cite{Nakkajima}, Cor.\ 3.4):

\begin{proposition}
\label{prop:crys}
There exist 
a formal scheme  $T=\Spf(\O)$, where  
$\O$  is a $p$-adically complete formally smooth $\Zp$-subalgebra  
of  $W(k)$  such that 
$\Oi:=\O\otimes_\Zp\Fp$  is a localication of an 
$\Fp$-subalgebra  of  $k$  of finite type, 
a model  $\Y$  of  $Y$  over  $T_1:=\Spec(\O_1$),  
such that, for any closed point  $t\in T_1$, 
we have canonical isomorphisms
$\Hicr(\Y/T)\otimes_{\O}W(\kt)\simeq\Hicr(\Yt/W(\kt))$  and 
$\Hicr(\Y/T)\otimes_{\O}W(k)\simeq\Hicr(Y/W(k))$. 
\end{proposition}

Note that the ring homomorphism  
$\O\to W(\kt)$  is the Teichm\"uller lift of the homomorphism 
$\Oi\to\kt$  corresponding to the morphism  $\{t\}\hookrightarrow T_1$.
Put
$\D=\Hicr(\Y/T)$  and let  
$K_0$  be the fraction field of  $W(k)$. 
By $p$-adic Hodge theory 
(\cite{Faltings}, \cite{Tsuji}, ...), we have  
$\D\otimes_{\O}K_0\simeq\Dcrys(\Hiet(\XKbar,\Qp))$.
By Th\'eor\`eme 1.2 of \cite{Chiarellotto-LeStum} 
(cf.\ also 
Thm.\ 1 of \cite{Katz-Messing} and 
Rem.\ 2.2 (4) of \cite{Nakkajima}), 
$\Hicr(\Yt/W(\kt))$  has pure weight  $i$. 

Thus, with all $\ell$-th factors together, 
$\rlL$  forms a geometric system of pure weight  $i$  somewhere 
in the sense of Definition \ref{def:system}.
\end{proof}

\section{Applications}\label{sec:application}

In this section, 
we give a sample of applications of Theorem \ref{thm:MT1} 
to Iwasawa theory 
(Similar generalizations should be possible with 
(the global versions of) the extensions studied in 
\cite{Ozeki}, but we restrict ourselves to our  $M/K$). 
Throughout this section, let  
$F$  be an algebraic number field (:= a finite extension of  $\Q$), 
$p$  a prime number, 
$\Fcyc$  the cyclotomic $\Zp$-extension of  $F$, and  
$M$  the extension of  $F$  obtained by adjoining 
all $p$-power roots of all elements of  $F$. 

The first application is to a version of 
Mazur's control theorem (\cite{Mazur}) for an extension  
$L/F$  contained in  $M$, which may be larger than  $\Fcyc$; 
this gives a particular case to which a theorem of Greenberg 
(\cite{Greenberg}) is applicable. 
Recall that, for an abelian variety  
$A$  over an algebraic numer field  $F$, we define 
$$
    \Selp(A/F)\ :=\ 
    \Ker\left(H^1(F,A[\pinf])\to\prod_v H^1(\Fv,A)[\pinf]\right),
$$
where the product is over all places  $v$  of  $F$, 
$\Fv$  is the completion of  $F$  at  $v$, 
and, if  $L/F$  is an algebraic extension, 
$$
    \Selp(A/L)\ :=\ \limiFp\Selp(A/F'),
$$
where the limit is taken with respect to the restriction maps when  
$F'$  runs over the finite extensions of  $F$  contained in  $L$. 

\begin{theorem}
\label{thm:control}
Let  $A$  be an abelian variety over  $F$  which has 
potentially good ordinary reduction at all places of  $F$  lying above  $p$. 
Let   $L$  be a $p$-adic Lie extension of  $F$  
which is contained in  $M$  and is unramified outside a 
finite set of places of  $F$. 
Then, for any finite extension  $F'/F$  contained in  $L$, 
the restriction map  
$$
    \sLFp:\ \Selp(A/F')\ \to\ \Selp(A/L)^{\Gal(L/F')}
$$ 
has finite kernel and cokernel. 
Furthermore, the order of  $\Ker(\sLFp)$  is bounded 
as  $F'$  varies.
\end{theorem}

\begin{proof}
This has been proved by Greenberg 
(\cite{Greenberg}, Thm.\ 1) 
if  $L/F$  is an admissible extension such that  
$A(L)[\pinf]$  is finite 
(here, a $p$-adic Lie extension is said to be 
{\it admissible} if, for any place  $v$  of  $F$  dividing  $p$, 
we have  
$\dv'=\iv'$, where
$\dv=\Lie(D_v)$  (resp.\ 
$\iv=\Lie(I_v)$)  
denotes the Lie algebra of a decomposition (resp.\ inertia) group 
of  $v$  in $\Gal(L/F)$, and  
$(\cdots)'$  denotes the derived Lie algebra). 
Since the finiteness of  
$A(L)[\pinf]$  is known by 
Theorem \ref{thm:MT1}, it remains for us to 
prove the admissibility of the extension  $L/F$. 
If the residue degree of  $L/F$  at  $v$  is finite, 
or equivalently, if  $I_v$  is of fintie index in  $D_v$, 
then we have  
$\dv=\iv$  and {\it a fortiori} 
$\dv'=\iv'$. 
This is indeed the case for our  $L/F$, since  
$M/F$  has finite residue degree at  $v$  by 
Lemma \ref{lem:residue}.
\end{proof}

Next we generalize some results of Kurihara and Hachimori 
to apply to the generalized Euler characteristic of 
fine Selmer groups. Let  
$E$  be an elliptic cure over  $F$, and let 
$\Sbad=\Sbad(E)$  be the set of finite places of  $F$  
at which  $E$  has bad reduction. Let  
$S$  be a finite set of places of  $F$  which contains  
$\Sbad\cup\Sp\cup\Sinf$, where  
$\Sp$  is the set of places of  $F$  lying above  $p$  and  
$\Sinf$  is the set of infinite places of  $F$. 
Let  $\FS$  denote the maximal Galois extension of  $F$  
unramified outside  $S$. Let  
$L/F$  be a $p$-adic Lie extension unramified outsie  $S$. Put  
$G=\Gal(L/F)$, 
$\LG=\Zp[\![G]\!]:=\limp_U\Zp[G/U]$, where  
$U$  runs over the open normal subgroups of  $G$, and  
$\IG:=\Ker(\LG\to\Zp)$  the augmentation ideal. Let  
$\TpE$  denote the $p$-adic Tate module of  $E$. Define 
$$
    \HoneL\ :=\ \limpFp H^1(\FS/F',\TpE),
$$
where  $F'$  runs over the finite extensions of  $F$  
contained in  $L$, and the limit is taken with respect to the 
correstriction maps. 
This  $\HoneL$  has a natural structure of continuous 
$\LG$-module. The following theorem is a variant of 
Kurihara's (\cite{Kurihara}, Lem.\ 4.3) and 
Hachimori's (\cite{Hachimori}, Thm.\ 1.3):

\begin{theorem}
\label{thm:rankH}
Assume  $p\geq 3$  and that  $F$  is an abelian extension of  $\Q$ 
which contains the $p$-th roots of unity. 
Let  $\alpha$  be an element of  $F^\times$  and let  
$L:=F(\mpinf,\alpha^{1/p^\infty})$  be the extension of  $k$  
obtained by adjoining all $p$-power roots of  
$1$  and  $\alpha$. Assume also that 
the elliptic curve  $E$  is defined over  $\Q$  and 
has good ordinary reduction at every  $v\in\Sp$. 
Then we have:

\sn
$(\ref{thm:rankH}.1)$ \ 
$\rankZp(\HoneL/\IG\HoneL)=\rankLG(\HoneL)$.

\sn
$(\ref{thm:rankH}.2)$ \
The homology group  
$H_i(G,\HoneL)$  is finite for any  $i\geq 1$.
\end{theorem}

\begin{proof}
For the brevity of exposition, let us first 
write down some conditions:

\sn
(\ref{thm:rankH}.3) \ 
$E(L)[\pinf]$  is finite.  

\sn
(\ref{thm:rankH}.4) \ 
$H^2(\FS/L,E[\pinf])=0$.

\sn
(\ref{thm:rankH}.5) \ 
$\Selp(E/L)\dual$  is a torsion $\LG$-module.

\sn
(\ref{thm:rankH}.6) \  
$\rankLG(\Selp(E/L)\dual)=\sum_{v\in\Spss}[\Fv:\Qp]$.

\sn
(\ref{thm:rankH}.7) \ 
$\rankLGam(\Selp(E/\Fcyc)\dual)=\sum_{v\in\Spss}[\Fv:\Qp]$.

\sn
Here, 
$\Spss$  is the set of places in  $\Sp$  at which  
$E$  has supersingular reduction, and  
$\LGam$  is the Iwasawa algebra of  $\Gamma=\Gal(\Fcyc/F)$.
By Kato (\cite{Kato}; cf.\ the discussions following 
Theorem 2.8 (pp.\ 451--452) of \cite{Hachimori-Venjakob}), the condition 
(\ref{thm:rankH}.7) holds if  $E/\Q$  has good reduction at  $p$. 
By Hachimori and Venjakob (\cite{Hachimori-Venjakob}, Thm.\ 2.8), 
we have 
(\ref{thm:rankH}.7) $\Rightarrow$ (\ref{thm:rankH}.6) 
(note that  $G=\Gal(L/F)$  does not contain $p$-torsion elements).  
By our assumption  $\Spss=\varnothing$, we have 
$\sum_{v\in\Spss}[\Fv:\Qp]=0$, and hence  (\ref{thm:rankH}.5) holds. 
By Hachimori and Venjakob (\cite{Hachimori-Venjakob}, Thm.\ 7.2), we have 
(\ref{thm:rankH}.3) + (\ref{thm:rankH}.5) $\Rightarrow$ (\ref{thm:rankH}.4), 
where (\ref{thm:rankH}.3) is true by our Theorem \ref{thm:MT1}. 
By Hachimori (\cite{Hachimori}, Thm.\ 1.3), we have 
(\ref{thm:rankH}.3) + (\ref{thm:rankH}.4) $\Rightarrow$ 
(\ref{thm:rankH}.1) + (\ref{thm:rankH}.2).
\end{proof}

Next, we give an application of Theorem \ref{thm:rankH}. 
For an elliptic cure  $E$  over  $F$  and a $p$-adic Lie extension  
$L/F$  contained in  $\FS$, we define the 
{\it fine Selmer group}  of  $E$  over  $L$  by 
$$
    \Rp(E/L)\ :=\ \limiFp\Rp(E/F'),
$$
where  
$$
    \Rp(E/F')\ :=\ 
	\Ker\left(H^1(\FS/F',E[\pinf])\to\bigoplus_{v'\in\SFp}H^1(\Fpvp,E[\pinf])\right).
$$
Here,   
$\SFp$  is the set of places of  $F'$  lying above  $S$. 

\begin{theorem}
\label{thm:fineSel}
Assume that  $p\geq 3$,  
$F/\Q$  is abelian, and that the 
elliptic curve  $E/F$  is defined over  $\Q$  and has 
good ordinary reduction at every  $v\in\Sp$. 
Let  $L/F$  be as in Theorem \ref{thm:rankH}. 
Assume further that:

\sn
$(\ref{thm:fineSel}.1)$ \  
$\Rp(E/F)$  is finite.

\sn
Then the following two conditions are equivalent:

\sn
$(\ref{thm:fineSel}.2)$ \ 
$\Rp(E/L)$  has finite $G$-Euler characteristic.

\sn
$(\ref{thm:fineSel}.3)$ \ 
The natural map
$$
    \HoneL/\IG\HoneL\ \to\ H^1(\FS/F,\TpE)
$$
has finite kernel.
\end{theorem}

\begin{proof}
The conclusion of this theorem has been proved 
by Hachimori (\cite{Hachimori}, Thm.\ 3.1)  
(cf.\ also    \cite{Coates-Sujatha})
under the assumption of 
(\ref{thm:fineSel}.1), 
(\ref{thm:rankH}.2), 
(\ref{thm:rankH}.4), and the following two conditions:

\sn
$(\ref{thm:fineSel}.4)$ \ 
$H^i(G,E(L)[\pinf])$  and  
$H^i(G_w,E(L_w)[\pinf])$  are finite for all  
$i\geq 0$  and  $w\in S_L$.

\sn
$(\ref{thm:fineSel}.5)$ \ 
$\rankZp(\HoneL/\IG\HoneL)=\rankZp(H^1(\FS/F,\TpE))$. 

\sn
The same theorem of Coates-Sujatha and Hachimori 
says that, under the assumption of  
(\ref{thm:fineSel}.1), 
(\ref{thm:rankH}.2) and 
(\ref{thm:rankH}.4), the last condition  
(\ref{thm:fineSel}.5)  is equivalent to 
(\ref{thm:rankH}.1). 
The truth of 
(\ref{thm:rankH}.4)  was shown in the proof of 
Theorem \ref{thm:rankH}, and the condition 
(\ref{thm:fineSel}.4)  holds by Theorem \ref{thm:MT1} 
(cf.\ \cite{SerreCG}, \S 4.1). 
Hence the theorem follows.
\end{proof}

Next we turn to the study of the generalized $G$-Euler characteristics  
of the usual Selmer groups. 
Let  $L$  be a $p$-adic Lie extension of  $F$  which contains  $\Fcyc$, 
and put  
$G=\Gal(L/F)$, 
$\Gamma=\Gal(\Fcyc/F)$, and  
$H=\Gal(L/\Fcyc)$. 
For a $p$-primary discrete $G$-module  $N$, let  
$(H^i(G,N),d_i)$  be the complex defined by 
$$
  d_i:\            H^i(    G,N) \ \overset{   f_i}\to\ 
        H^0(\Gamma,H^i(    H,N))\ \overset{\phi_i}\to\ 
        H^1(\Gamma,H^i(    H,N))\ \overset{   g_i}\to\ 
                   H^{i+1}(G,N),
$$
where  
$f_i$  (resp.\ $g_i$)  is the restriction (resp.\ inflation) map and  
$\phi_i$  is the natural map  
$H^i(H,N)^\Gamma\to H^i(H,N)_\Gamma$. Let  
$\hi=\hi(G,N)$  be the $i$-th cohomology group 
of this camplex. 
We say that the $G$-module  $N$  has 
{\it finite generalized $G$-Euler characteristic} if  
$\hi$  is finite for all  $i\geq 0$  and  
$\hi=0$  for all but finitely many  $i$. 
If this is the case, we define
$$
    \chi(G,N)\ :=\ \prod_{i\geq 0}(\#\hi)^{(-1)^i}.
$$
Recall that we define 
$$
    \ShaEF\ :=\ \Ker\left(H^1(F,E)\to\prod_v H^1(\Fv,E)\right).
$$

\begin{theorem}
\label{thm:Selmer}
Assume  $p\geq 5$. 
Let  $L$  be a $p$-adic Lie extension of  $F$  such that  
$M\supset L\supset\Fcyc$  and  
$L/F$  is unramified outside a finite set of places of  $F$. Put  
$G=\Gal(L/F)$, 
$H=\Gal(L/\Fcyc)$  and  
$\Gamma=\Gal(\Fcyc/F)=G/H$. 
Assume  $G$  contains no $p$-torsion elements. Let  
$E$  be an elliptic curve over  $F$  which has good ordinary 
reduction at every place  $v\in\Sp$. 
Assume further the following two conditions:

\sn
$(\ref{thm:Selmer}.1)$ \ 
$\ShaEF[\pinf]$  is finite.

\sn
$(\ref{thm:Selmer}.2)$ \ 
$\Selp(E/L)\dual/\Selp(E/L)\dual[\pinf]$  
is finitely generated as a $\LH$-module.

\sn
Then  
$\Selp(E/L)$  has finite generalized $G$-Euler characteristic 
if and only if  
$\Selp(E/\Fcyc)$  has finite 
$\Gamma$-Euler characteristic. 
If this is the case, we have 
$$
    \chi(G,\Selp(E/L))\ =\ 
    \chi(\Gamma,\Selp(E/\Fcyc))\times
	\left|\prod_{v\in T}L_v(E,1)\right|,
$$
where  $T$  is the set of finite places  $v\nmid p$  of  $F$  
at which the inertia group in  $G$  has infinite order, and  
$L_v(E,s)$  is the $v$-th Euler factor of the $L$-function of  $E/F$. 
\end{theorem}

\begin{proof}
This has been proved by Zerbes 
(\cite{Zerbes}, Thm.\ 1.1) under the additional assumption that  
$H^i(H,E(L)[\pinf])$  and  
$H^i(H_w,E_v(\kappa_w))$  are finite for all  
$i\geq 0$  and  $v\in\Sp$, where  
$H_w$  is the decomposition group in  $H$  of a place  
$w$  of  $L$  lying above  $v$, 
$E_v$  is the reduction of  $E$  mod $v$, and  
$\kappa_w$  is the residue field of  $w$. 
These conditions hold because  
$E(L)[\pinf]$  is finite by Theorem \ref{thm:MT1} and
$E_v(\kappa_w)[\pinf]$  is finite since  $\kappa_w$  is finite 
by Lemma \ref{lem:residue}.
\end{proof}

Next we address the question of the existence 
of non-trivial pseudo-null $\LG$-submodules 
of Selmer groups. 
Recall that a continuous $\LG$-module  $N$  is said to be 
{\it pseudo-null} if  
$\Ext_{\LG}^i(N,\LG)=0$  for  $i=0,1$  
(\cite{Venjakob}).

\begin{theorem}
\label{thm:pseudo-null}
Assume  $p\geq 3$. 
Let  $L$  be a $p$-adic Lie extension of  $F$  
contained in  $M$  and unramified outside a finite set 
of places of  $F$. 
Assume  
$G=\Gal(L/F)$  has no $p$-torsion elements. 
Let  $E$  be an elliptic curve over  $F$  
which has good reduction at every  $v\in\Sp$. 
Assume further the following two conditions:

\sn
$(\ref{thm:pseudo-null}.1)$ \ 
$\dim(G_v)\geq 2$  for any finite place  $v$  of  $F$  at which  
$E$  has bad reduction. 
\\
(Here, $G_v$  is a decomposition group of  $v$  in  $G$.)

\sn
$(\ref{thm:pseudo-null}.2)$ \ 
$\Selp(E/L)\dual$  is a torsion $\LG$-module.

\sn
Then  $\Selp(E/L)\dual$  contains no non-trivial pseudo-null 
$\LG$-submodules. 
\end{theorem}

\begin{proof}
This follows from Theorem 3.2 of \cite{Hachimori-Ochiai}, 
with the help of Theorem 7.2 of \cite{Hachimori-Venjakob} and our 
Theorem \ref{thm:MT1}. 
(For  $\dim(G_v)\geq 2$  at  $v\in\Sp$, see 
Lemma 3.9 of \cite{Hachimori-Venjakob}. 
Note also that the residue fields of $L$  at places 
above  $p$  are finite (Lem.\ \ref{lem:residue})  and that  
$L/F$  is deeply ramified at  $v\in\Sp$  
(\cite{Coates-Greenberg}, Thm.\ 2.13).
\end{proof}

\begin{remark}
We have the same conclusion as in the above theorem 
without assuming (\ref{thm:pseudo-null}.2) if
\\
-- $F$  is abelian over  $\Q$  and contains  $\Q(\mp)$, 
\\
-- $L\supset\Fcyc$  and  $\dim(G)=2$, and 
\\
-- $E$  is defined over  $\Q$  and has good ordinary reduction 
at all  $v\in\Sp$.
\\
This is because one can show (\ref{thm:pseudo-null}.2) 
as in the proof of Theorem \ref{thm:rankH}.
\end{remark}


\medskip\noindent
(Y.\ K.)\\
Seika Girls' High School \\
4-19-1, Sumiyoshi, Hakata-ku, Fukuoka, 812-0018 Japan \\
Email address: {\tt hechitinukibire@yahoo.co.jp}

\sn
(Y.\ T.)\\
Faculty of Mathematics, Kyushu University \\
744, Motooka, Nishi-ku, Fukuoka, 819-0395 Japan\\
Email address: {\tt taguchi@math.kyushu-u.ac.jp}

\end{document}